\documentclass{article}

\usepackage[a4paper]{geometry}
\usepackage{amsmath}
\usepackage{amssymb}
\usepackage{amsthm}
\usepackage{enumerate}
\usepackage{color}
\usepackage{url}
\usepackage{tikz}
\usepackage{array}

\allowdisplaybreaks[1]

\newtheorem{theorem}{Theorem}

\newtheorem{definition}[theorem]{Definition}
\newtheorem{example}[theorem]{Example}

\newcommand{\Z}{\mathbb{Z}}
\newcommand{\F}{\mathbb{F}}

\long\def\symbolfootnote[#1]#2{\begingroup%
\def\thefootnote{\fnsymbol{footnote}}\footnote[#1]{#2}\endgroup}

\begin{document}

\title{Costas cubes}
\author{Jonathan Jedwab \and Lily Yen}
\date{16 February 2017 (revised 3 August 2017)}
\maketitle

\symbolfootnote[0]{
J.~Jedwab is with Department of Mathematics, 
Simon Fraser University, 8888 University Drive, Burnaby BC V5A 1S6, Canada.
\par
L.~Yen is with Department of Mathematics and Statistics,
Capilano University, 2055 Purcell Way, North Vancouver BC V7J 3H5, Canada
and Department of Mathematics, 
Simon Fraser University, 8888 University Drive, Burnaby BC V5A 1S6, Canada.
\par
J.~Jedwab is supported by NSERC.
\par
Email: {\tt jed@sfu.ca}, {\tt lyen@capilanou.ca}
\par
}

\begin{abstract}
A Costas array is a permutation array for which the vectors joining pairs of $1$s are all distinct. 
We propose a new three-dimensional combinatorial object related to Costas arrays: an order $n$ \emph{Costas cube} is an array $(d_{i,j,k})$ of size $n \times n \times n$ over $\Z_2$ for which each of the three projections of the array onto two dimensions, namely 
$(\sum_i d_{i,j,k})$ and $(\sum_j d_{i,j,k})$ and $(\sum_k d_{i,j,k})$,
is an order $n$ Costas array.
We determine all Costas cubes of order at most $29$, showing that Costas cubes exist for all these orders except $18$ and $19$ and that a significant proportion of the Costas arrays of certain orders occur as projections of Costas cubes. We then present constructions for four infinite families of Costas cubes. 
\end{abstract}

\section{Introduction}
\label{sec:introduction}
We write $I[X]$ for the indicator function of condition~$X$ (so $I[X] = 1$ if $X$ is true, and 0 otherwise).
Let $\sigma \in \mathcal{S}_n$ be a permutation on $\{1,2,\dots,n\}$. The \emph{permutation array} $(s_{i,j})$ corresponding to $\sigma$ is the $n \times n$ array given by $s_{i,j} =I[ \sigma(j) = i]$, where the indices $i$ and $j$ range over $\{1,2,\dots,n\}$.
For example, representing index $i$ as increasing from left to right and index $j$ as increasing from bottom to top, the permutation array corresponding to the permutation $(3,5,4,2,6,1) \in S_6$ is
\begin{center}
\vspace{1em}
\begin{tikzpicture}[yscale=-1,scale=0.5]
   \foreach \i/\j in {1/1, 2/3, 3/6, 4/4, 5/5, 6/2}
      \fill[black!20] (\i,\j) rectangle +(1,1);
   \draw (1,1) grid (7,7);
   \foreach \i in {1,...,6} {
      \node[above] at (\i+1/2,8) {$\i$};
      \node[left] at (1,7-\i+1/2) {$\i$};
   \node[above] at (4.5,9) {$i$};
   \node[left] at (0,3.5) {$j$};
   }
\end{tikzpicture}
\end{center}
where 1 entries of the permutation array are represented as shaded squares.

A permutation array $(s_{i,j})$ of order $n$ is a \emph{Costas array} if the vectors formed by joining pairs of 1s in $(s_{i,j})$ are all distinct. J.P.\ Costas introduced these arrays in $1965$ in order to improve the performance of radar and sonar systems \cite{DrakOpen11}: the radar or sonar frequency $f_i$ is transmitted in time interval $t_j$ if and only if $s_{i,j}=1$. An equivalent definition of a Costas array is a permutation array each of whose out-of-phase aperiodic autocorrelations is at most~$1$.

Each Costas array belongs to an equivalence class formed by its orbit under the action of the dihedral group~$D_4$ (the symmetry group of a square under rotation and reflection). The equivalence class of a Costas array of order greater than 2 has size four or eight, depending on whether or not its elements have reflective symmetry about a diagonal.

In 2008, Drakakis \cite{Drak08Higher} proposed a generalization of Costas arrays to dimensions other than two, based on aperiodic autocorrelations, and gave further details in~\cite{Drak10Higher}. This viewpoint has the advantage that the one-dimensional case corresponds to a Golomb ruler \cite{Babcock}, and was subsequently studied in~\cite{NazarioOrtiz12}. However, the associated generalization of a permutation to more than two dimensions is problematic when the number of dimensions is odd, and the classical constructions of Costas arrays due to Gilbert-Welch and Golomb (see Theorems~\ref{thm:gilbertwelch} and~\ref{thm:golomb}) do not seem to generalize in a natural way.

We instead propose a different generalization of Costas arrays to three dimensions, which depends directly on two-dimensional Costas arrays.
\begin{definition}
The \emph{projections} of a three-dimensional array $(d_{i,j,k})$ are Projection $A =(a_{i,j}) = (\sum_k d_{i,j,k})$, Projection $B= (b_{i,k}) =(\sum_j d_{i,j,k})$, and Projection $C = (c_{j,k}) = (\sum_i d_{i,j,k})$.
\end{definition}

We call a multi-dimensional array whose entries all lie in $\{0,1\}$ an \emph{array over $\Z_2$}.
\begin{definition}
\label{defn:Costascube}
An order $n$ \emph{Costas cube} is an $n\times n \times n$ array over $\Z_2$ for which Projections $A$, $B$, $C$ are each order $n$ Costas arrays.
\end{definition}

For example, let $D = (d_{i,j,k})$ be the $6\times6\times6$ array given by
\[
d_{i,j,k} = I\big[(i,j,k) \in \{(1,6,4),\, (2,4,6),\, (3,1,2),\, (4,3,1),\, (5,2,5),\, (6,5,3)\}\big].
\]
Then $D$ is an order 6 Costas cube, and the Costas permutations corresponding to Projections $A$, $B$, $C$ are $(3,5,4,2,6,1)$, $(4,3,6,1,5,2)$, $(3,1,5,6,2,4)$, respectively. The Costas cube $D$ and its three associated projections are shown in Figure~\ref{fig:cube}, where 1 entries of $D$ are represented by shaded cubes of size $1 \times 1 \times 1$.  
(A three-dimensional array that is a Costas cube according to Definition~\ref{defn:Costascube} is also a Costas cube according to the definition of Drakakis \cite{Drak08Higher}, \cite{Drak10Higher}: if the multiset of vectors joining pairs of 1 entries in such an array contains a repeat, so does the multiset of vectors joining pairs of 1 entries for each of its three projections.)

\begin{figure}[htbp]
\begin{center}
\includegraphics[width=0.8\textwidth]{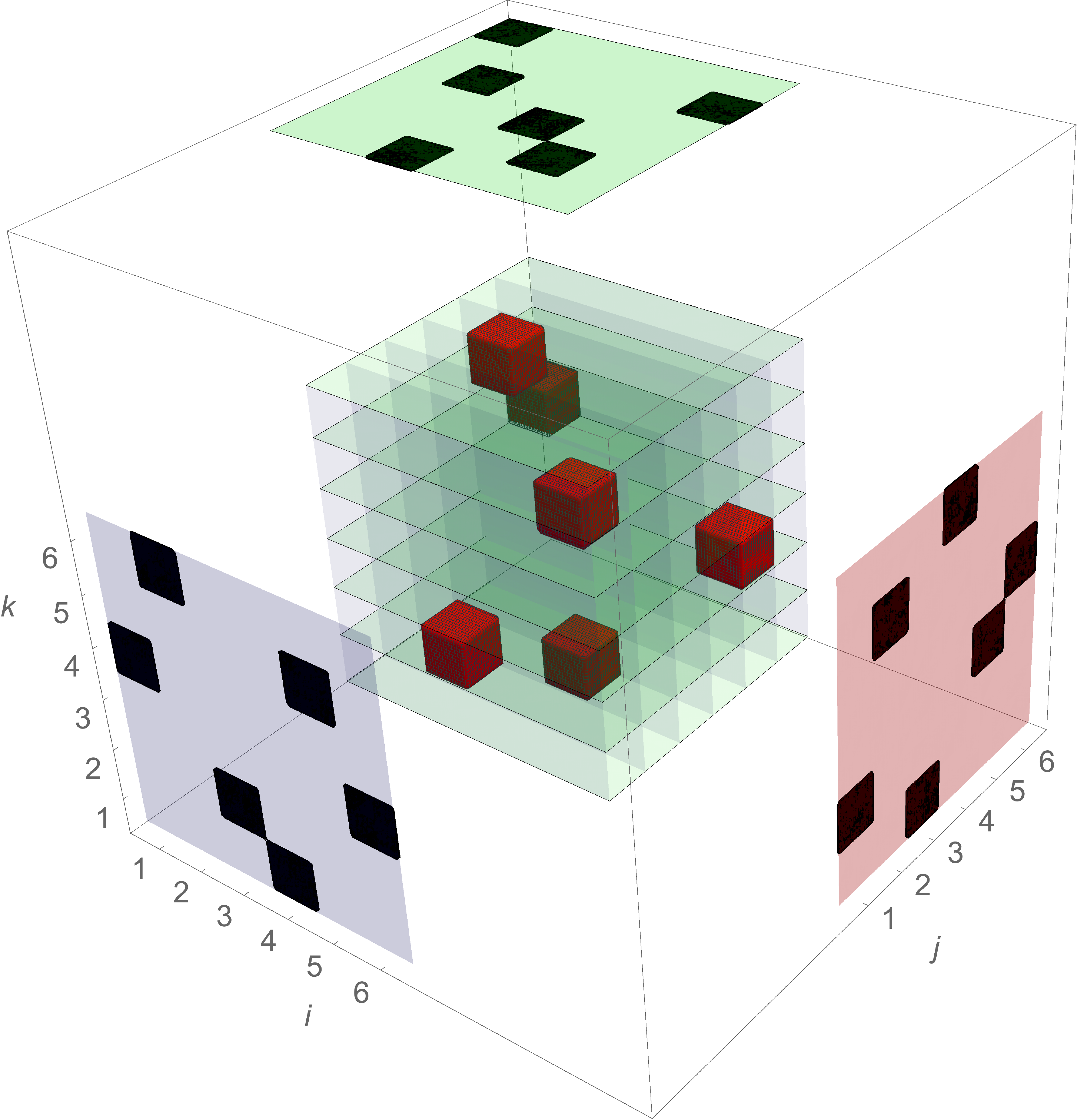}
\vspace{2em} \\
\includegraphics[width=0.3\textwidth]{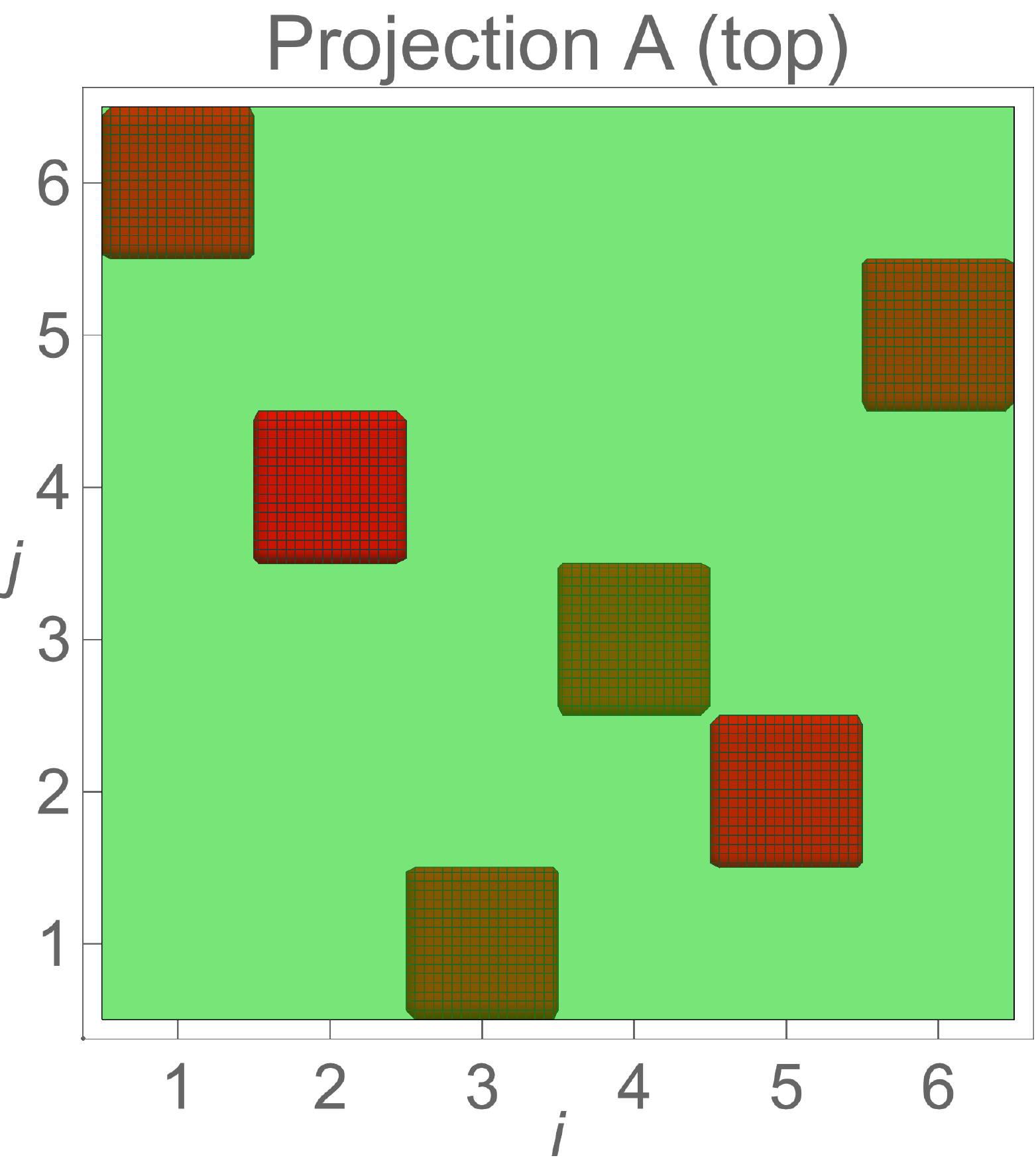} \hspace{1em}
\includegraphics[width=0.3\textwidth]{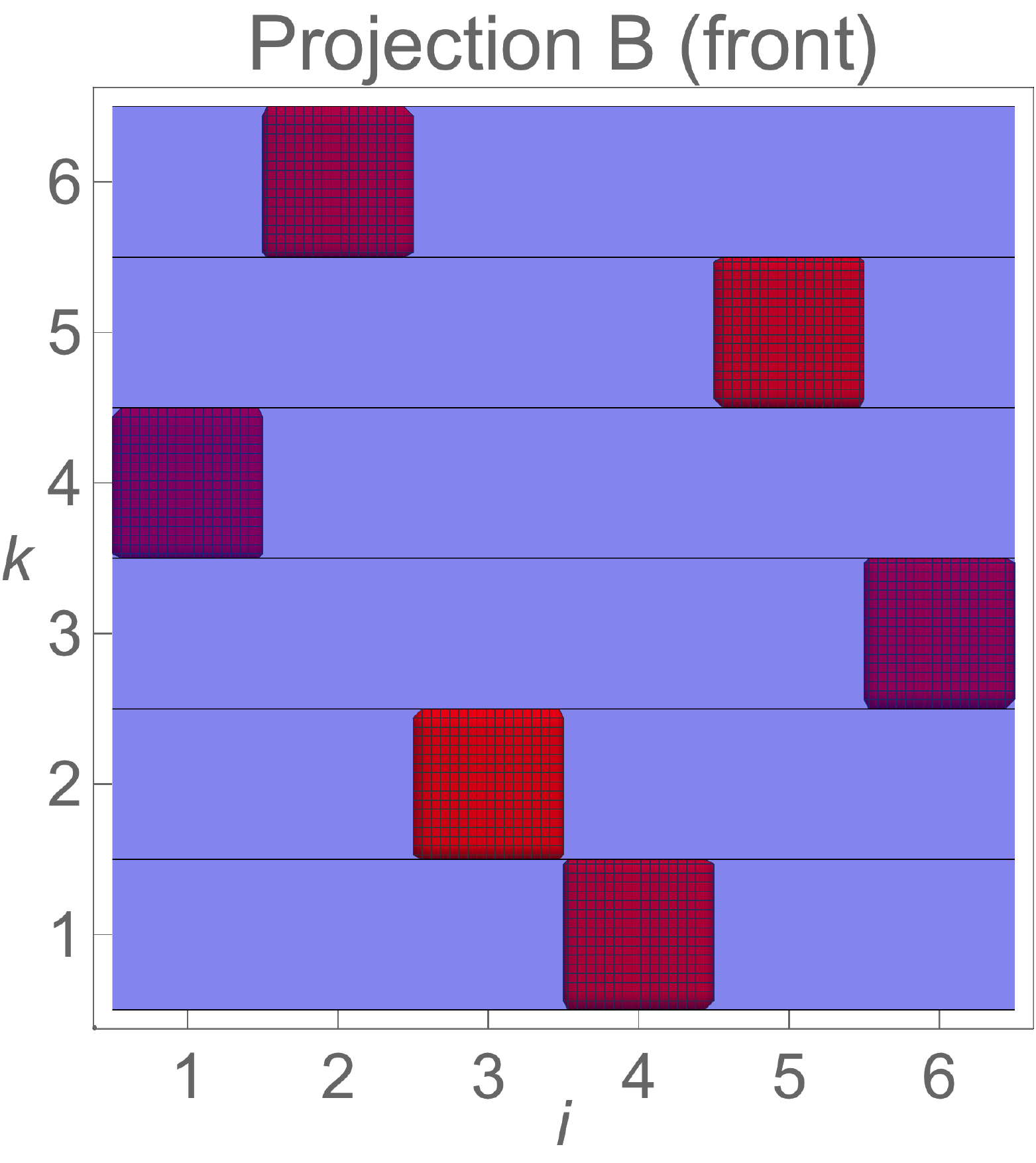} \hspace{1em}
\includegraphics[width=0.3\textwidth]{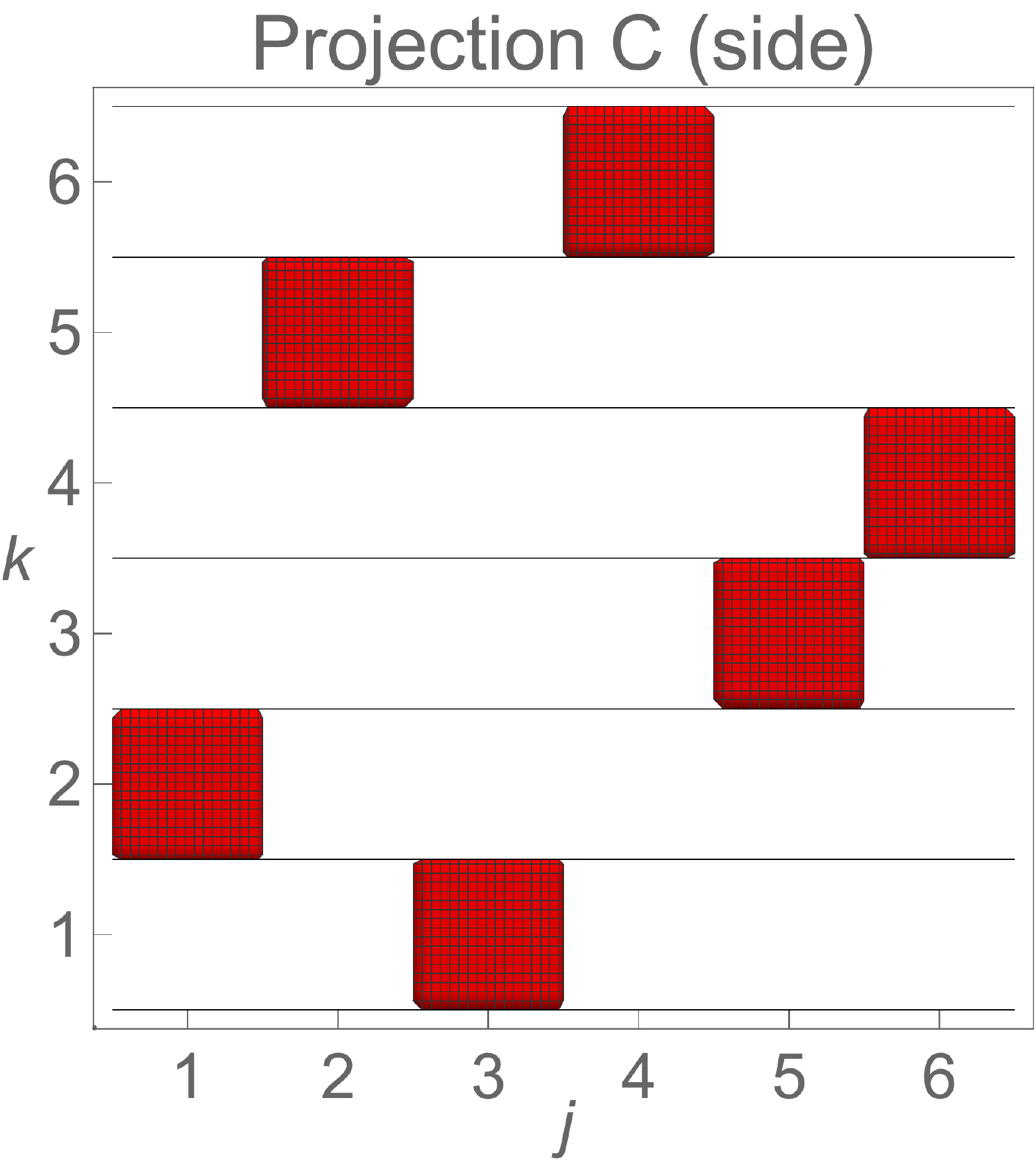}
\caption{Costas cube and its three projections}
\label{fig:cube}
\end{center}
\end{figure}

Before considering Costas cubes in more detail, we explain how Definition~\ref{defn:Costascube} can be formulated in terms of a three-dimensional generalization of permutation arrays. Define an order $n$ \emph{permutation cube} to be an $n \times n\times n$ array over $\Z_2$ for which Projections $A$, $B$, $C$ are each order~$n$ permutation arrays. Then a Costas cube is a permutation cube for which the permutation arrays given by Projections $A$, $B$, $C$ have the additional property that they are Costas arrays. Moreover, we can regard an order $n$ permutation array $(s_{i,j})$ as an $n \times n$ array over $\Z_2$ for which $\sum_{i} s_{i,j} = 1$ for each $j$, and $\sum_{j} s_{i,j} = 1$ for each $i$. We then see that an order $n$ permutation cube $(d_{i,j,k})$ can be equivalently defined as an $n\times n \times n$ array over $\Z_2$ for which each two-dimensional subarray contains exactly one $1$ entry: $\sum_{i,j} d_{i,j,k}=1$ for each $k$, and 
$\sum_{i,k} d_{i,j,k}=1$ for each $j$, and
$\sum_{j,k} d_{i,j,k}=1$ for each~$i$.
Eriksson and Linusson \cite{ErikssonLinusson00} refer to this equivalent definition of a permutation cube as a \emph{sparse} 3-dimensional permutation array, and note that sparse higher-dimensional arrays were used by Pascal in 1900 \cite{Pascal1900} to define higher-dimensional determinants. 

Each Costas cube $D$ belongs to an equivalence class $E(D)$ formed by its orbit under the action of the order 48 symmetry group of a cube under rotation and reflection; the subgroup of this symmetry group under rotation but not reflection has order 24 and is isomorphic to~$S_4$. By taking Projection $A$ (say) of each of the elements of $E(D)$, and discarding repeats if any, we obtain the set $S(D)$ of distinct Costas arrays occurring as projections of~$D$. 
This set $S(D)$ is the union of one or more equivalence classes of Costas arrays, and so its size is a multiple of~$4$.
We may exclude reflections of $D$ when forming $S(D)$, because a reflection of a projection of a Costas cube can be realized as a rotation of the cube. The size of $S(D)$ is therefore at most~24, although it can be smaller. For example, we find that as $D$ ranges over the order 6 Costas cubes (as determined by the method of Section~\ref{sec:determination}), the size of the set $S(D)$ takes each value in $\{4,8,12,16,20,24\}$. In particular, the set $S(D)$ for the order 6 Costas cube $D = (d_{i,j,k})$ given by 
\[
d_{i,j,k} = I\big[(i,j,k) \in \{(1,2,4), (2,4,1), (3,5,6), (4,1,2), (5,6,3), (6,3,5)\}\big]
\]
has size 4: its elements comprise a single equivalence class of Costas arrays whose corresponding permutations are
$(2,4,5,1,6,3)$, $(3,6,1,5,4,2)$, $(4,1,6,2,3,5)$, $(5,3,2,6,1,4)$.
We show in Section~\ref{sec:infinite} that this example is a member of an infinite family of Costas cubes all of whose elements $D$ satisfy $|S(D)|=4$.

We have three principal motivations for proposing Costas cubes. The first is to provide new perspectives on the observed existence pattern for Costas arrays. The second is to ask whether the favourable projection and autocorrelation properties of Costas cubes render them suitable for use in digital communications applications such as optical orthogonal codes and digital watermarking (as has been proposed \cite{NazarioOrtiz12} for the generalization of Costas arrays due to Drakakis). The third is to present these structures as being of mathematical interest in their own right.

\section{Determination of Costas cubes of order at most 29}
\label{sec:determination}

All Costas arrays of order at most 29 have been determined by exhaustive search: those of order at most 27 were listed in the database \cite{RickardDatabase}, and those of order 28 and 29 are listed in \cite{Costas_Enum_28} and \cite{Costas_Enum_29}, respectively.
Now any two of the Projections $A$, $B$, $C$ of a permutation cube determine the cube and therefore the third Projection. 
We may therefore determine all Costas cubes of order $n \le 29$ in the following way. For each ordered pair of (not necessarily distinct, not necessarily inequivalent) Costas arrays ($A$, $B$) of order $n$, let $D$ be the permutation cube whose Projections $A$ and $B$ are arrays $A$ and $B$, respectively, and retain those permutation cubes $D$ for which Projection $C$ is a Costas array. All retained permutation cubes $D$ are Costas cubes of order~$n$; select one representative of each equivalence class of retained cubes.

Table~\ref{tab:inequivcubes} displays, for each $n \le 29$:
the number of equivalence classes of Costas cubes of order~$n$;
the number of equivalence classes of Costas arrays of order $n$ which are projections of some Costas cubes of order $n$; and, for comparison, the total number of equivalence classes of Costas arrays of order $n$.
We see that Costas cubes exist for all orders $n \le 29$ except $18$ and $19$, and that a significant proportion of the Costas arrays of certain orders occur as projections of Costas cubes. 

\begin{table}[h]
\centering
\begin{tabular}{|c|c|c|c|}
\hline
Order 		&\# equivalence	classes & \# equivalence classes& Total \# 		\\ 
\hspace{8em}	& of Costas cubes 	& of Costas arrays	& equivalence classes 	\\
		& 			& which are projections	& of Costas arrays	\\
		&			& of some Costas cube	& 			\\ \hline
2		& 1			& 1			& 1 			\\
3		& 1			& 1			& 1 			\\
4		& 2			& 1			& 2			\\
5		& 13			& 6			& 6			\\
6		& 47			& 17			& 17			\\
7		& 30			& 26			& 30			\\
8		& 42			& 44			& 60			\\
9		& 46			& 61			& 100			\\
10		& 69			& 133			& 277			\\
11		& 66			& 126			& 555			\\
12		& 34			& 74			& 990			\\
13		& 11			& 22			& 1616			\\
14		& 6			& 6			& 2168			\\
15		& 33			& 19			& 2467			\\
16		& 6			& 6			& 2648			\\
17		& 19			& 12			& 2294			\\
18		& 0			& 0			& 1892			\\
19		& 0			& 0			& 1283			\\
20		& 2			& 3			& 810			\\
21		& 50			& 20			& 446			\\
22		& 4			& 9			& 259			\\
23		& 11			& 7			& 114			\\
24		& 2			& 1			& 25			\\
25		& 20			& 7			& 12			\\
26		& 1			& 2			& 8			\\
27		& 77			& 27			& 29			\\
28		& 3			& 4		 	& 89  			\\
29		& 33			& 18			& 23			\\ \hline
\end{tabular}
\caption{Inequivalent Costas cubes and their inequivalent projections to Costas arrays}
\label{tab:inequivcubes}
\end{table}

\section{Four infinite families of Costas cubes}
\label{sec:infinite}

In this section we give algebraic constructions for four infinite families of Costas cubes. 

Theorems~\ref{thm:gilbertwelch} and~\ref{thm:golomb} describe two classical constructions producing infinite families of Costas arrays.
In these theorems (and also in Theorems~\ref{thm:welch2} and~\ref{thm:golomb3} below), the equation appearing in the argument of the indicator function is regarded over the associated field ($\F_p$ or $\F_q$). 

\begin{theorem}[Gilbert-Welch construction $W_1(p,\phi,c)$ \cite{Gilbert}, \cite{Golomb84}]
\label{thm:gilbertwelch}
Let $p > 2$ be prime, let $\phi$ be a primitive element of $\F_p$, and let $c \in \F_p$.
Then the array $(s_{i,j})$ given by
\[
s_{i,j} =I[\phi^{j+c} = i ] \quad \mbox{for $i, j \in \{1, 2, \dots, p-1\}$} 
\]
is an order $p-1$ Costas array.
\end{theorem}

\begin{theorem}[Golomb construction $G_2(q,\phi,\rho)$ \cite{Golomb84}]
\label{thm:golomb}
Let $q > 3$ be a prime power, and let $\phi$ and $\rho$ be (not necessarily distinct) primitive elements of $\F_q$. Then the array $(s_{i,j})$ given by
\[
s_{i,j} = I[ \phi^i  + \rho^j = 1 ] \quad \mbox{for $i, j \in \{1, 2, \dots, q-2\}$} 
\]
is an order $q-2$ Costas array.
\end{theorem}

Several variants of the constructions of Theorems~\ref{thm:gilbertwelch} and~\ref{thm:golomb} have been found. Of interest in the present context are the variant family of Gilbert-Welch Costas arrays of Theorem~\ref{thm:welch2}, as described in \cite[Theorem~7.30]{Drak06Review}, and the variant family of Golomb Costas arrays of Theorem~\ref{thm:golomb3}.

\begin{theorem}[Gilbert-Welch construction $W_2(p,\phi)$]
\label{thm:welch2}
Let $p >3$ be prime and let $\phi$ be a primitive element of $\F_p$. 
Then the array $(s_{i,j})$ given by 
\[
s_{i,j} = I[ i =\phi^{j} -1] \quad \mbox{for $i, j \in \{1, 2, \dots, p-2\}$} 
\]
is an order $p-2$ Costas array.
\end{theorem}

\begin{theorem}[Golomb construction $G_3(q,\phi)$ \cite{Golomb84}]
\label{thm:golomb3}
Let $q > 3$ be a prime power, and let $\phi$ be a primitive element of $\F_q$ for which $1-\phi$ is also primitive. Then the array $(s_{i,j})$ given by
\[
s_{i,j} = I[ \phi^{i+1}  + (1-\phi)^{j+1} = 1 ] \quad \mbox{for $i, j \in \{1, 2, \dots, q-3\}$} 
\]
is an order $q-3$ Costas array.
\end{theorem}
For each prime power~$q$, there is a primitive element $\phi$ for which $1-\phi$ is also primitive (as required in Theorem~\ref{thm:golomb3}) \cite{CohenMullen91},~\cite{MorenoSotero90}. For each such $\phi$, the $G_2(q,\phi,1-\phi)$ Costas array $(t_{i,j})$ satisfies $t_{1,1}=1$ and the $G_3(q,\phi)$ Costas array $(s_{i,j})$ of Theorem~\ref{thm:golomb3} can be viewed as arising from the removal of the row and column of~$(t_{i,j})$ containing the position $(1,1)$ to form $(s_{i,j})=(t_{i+1,j+1})$.

We now give our first algebraic construction of an infinite family of Costas cubes. Each of Projections $A$, $B$, $C$ in this construction is a $G_2$ Golomb Costas array. For $q$ a prime power, we have
\begin{equation}
(1-y)^{-1}+(1-y^{-1})^{-1}=1 \quad \mbox{for $y \in \F_q \setminus \{0,1\}$}
\label{eq:identity}
\end{equation}
and 
\begin{equation}
\{ \phi^i : 1 \le i \le q-2 \} =  \F_q \setminus \{0,1\} \quad 
		\mbox{for $\phi \in \F_q$ primitive}.
\label{eq:star}
\end{equation}

\begin{theorem}
\label{thm:gggcube}
Let $q > 3$ be a prime power, and let $\phi$, $\rho$, $\psi$ be (not necessarily distinct) primitive elements of $\F_q$. Then the array $D = (d_{i,j,k})$ over $\Z_2$ given by 
\begin{equation}
  d_{i,j,k} = I[
  \phi^i + \rho^{-j} = 1
  \quad\mbox{and}\quad 
  \phi^{-i} + \psi^{k} = 1
  \quad\mbox{and}\quad 
  \rho^j + \psi^{-k} = 1
]
\quad \mbox{for $i$, $j$, $k \in \{1, 2, \dots, q-2\}$}
\label{eq:golomb a}
\end{equation}
is an order $q-2$ Costas cube for which 
Projection $A$ is a $G_2(q, \phi, \rho^{-1})$ Golomb Costas array, 
Projection $B$ is a $ G_2(q, \phi^{-1}, \psi)$ Golomb Costas array, and 
Projection $C$ is a $ G_2(q,\rho, \psi^{-1})$ Golomb Costas array.
\end{theorem}

\begin{proof}
By Definition~\ref{defn:Costascube}, it is sufficient to show that each of Projections $A$, $B$ and $C$ is the specified Costas array.

Consider Projection $A = (a_{i,j}) =  (\sum_{k} d_{i,j,k})$.
By \eqref{eq:identity}, any two of the three conditions in the indicator function of \eqref{eq:golomb a} (namely $\phi^i + \rho^{-j} = 1$ and $\phi^{-i} + \psi^{k} = 1$ and $\rho^j + \psi^{-k} = 1$) imply the third, so we may rewrite \eqref{eq:golomb a} as
\[
  d_{i,j,k} = I[
  \phi^i + \rho^{-j} = 1
  \quad\mbox{and}\quad 
  \phi^{-i} + \psi^{k} = 1
].
\]
For $i$, $j \in \{1, 2, \dots, q-2\}$, sum over $k$ to give
\[
a_{i,j} = \sum_{k=1}^{q-2} 
I[\phi^i + \rho^{-j} = 1 \quad\mbox{and}\quad \phi^{-i} + \psi^{k} = 1].
\]
Now $1-\phi^{-i} \in \F_q \setminus \{0,1\}$, so by \eqref{eq:star} there is exactly one value $k \in \{1, 2, \dots, q-2\}$ for which $\phi^{-i}+\psi^k =1$. Therefore
\[
a_{i,j} = I[ \phi^i + \rho^{-j} =1],
\]
and so Projection $A$ is a $G_2(q, \phi, \rho^{-1})$ Golomb Costas array by Theorem~\ref{thm:golomb}.

Similar arguments apply to Projections
$B = (b_{i,k}) = (\sum_j d_{i,j,k})$ and 
$C = (c_{j,k}) = (\sum_i d_{i,j,k})$.
\end{proof}

\begin{example}
We construct a Costas cube $D = (d_{i,j,k})$ of order $14$ according to Theorem~\ref{thm:gggcube}, using $q=16$. Represent $\F_{2^4}$ as $ \Z_2[x]/\langle1 + x^3 + x^4\rangle$, and take $\phi = x$ and $\rho = 1 + x^2 + x^3$ and $\psi = x + x^2 + x^3$. Then from \eqref{eq:golomb a}, the triples $(i,j,k)$ for which $d_{i,j,k} = 1$ are given by
\[
\begin{array}{*{16}{c}}
$i$ & &1&2&3&4&5&6&7&8&9&10&11&12&13&14\\
$j$ & &3&6&1&12&10&2&7&9&8&5&11&4&13&14\\
$k$ & &7&14&2&13&10&4&12&11&1&5&6&8&3&9
\end{array}
\]
and the permutations corresponding to Projections $A$, $B$, $C$ of $D$ are given by
\[
\begin{array}{*{16}{c}}
$A$ & &3&6&1&12&10&2&7&9&8&5&11&4&13&14\\
$B$ & &9&3&13&6&10&11&1&12&14&5&8&7&4&2\\
$C$ & &8&1&13&2&5&11&3&4&14&10&9&7&12&6
\end{array}
\]
\end{example}

Note that in the special case $\phi = \rho = \psi$, the three projections of the Costas cube~$D$ specified in Theorem~\ref{thm:gggcube} are $G_2(q,\phi,\phi^{-1})$, $G_2(q, \phi^{-1},\phi)$ and $G_2(q,\phi,\phi^{-1})$. From Theorem~\ref{thm:golomb}, these projections all belong to a single equivalence class of Costas arrays having symmetry about a diagonal, and so the set $S(D)$ of distinct Costas arrays occurring as projections of $D$ has size~$4$. The order 6 Costas cube with this property that was given at the end of Section~\ref{sec:introduction} is constructed by representing $\F_{2^3}$ as $\Z_2[x]/\langle1 + x^2 + x^3\rangle$ and taking $\phi = \rho = \psi = 1+x+x^2$.

We now give our second algebraic construction of an infinite family of Costas cubes. Projections $A$ and $B$ in this construction are (equivalent to) $W_2$ Gilbert-Welch Costas arrays, and Projection $C$ is a $G_2$ Golomb Costas array.

\begin{theorem}
\label{thm:gw2w2cube}
Let $p>3$ be prime, and let $\phi$, $\psi$ be (not necessarily distinct) primitive elements of $\F_p$. 
Then the array $D = (d_{i,j,k})$ over $\Z_2$ given by 
\begin{equation}
  d_{i,j,k} = I[i = \phi^j-1 = -\psi^k ] \quad \mbox{for $i$, $j$, $k \in \{1, 2, \dots, p-2\}$}
  \label{eq:golomb-welch}
\end{equation}
is an order $p-2$ Costas cube for which Projection $A$ is a $W_2(p, \phi)$ Gilbert-Welch Costas array, Projection $B$ is the reflection through a vertical axis of a $W_2(p, \psi)$ Gilbert-Welch Costas array, and Projection $C$ is a $G_2(p, \phi,  \psi)$ Golomb Costas array.
\end{theorem}

\begin{proof}
Consider Projection $A = (a_{i,j}) = (\sum_k d_{i,j,k})$.
For $i, j \in \{1, 2,  \dots, p-2\}$, sum \eqref{eq:golomb-welch} over $k$ to obtain
\[
 a_{i,j}= \sum_{k=1}^{p-2}I[ i = \phi^j-1\quad \text{and} \quad i = -\psi^k ].
\]
By \eqref{eq:star} we have $-\psi^k \in \F_p \setminus \{0,-1\}$, and so there is exactly one value $k \in \{1, 2,  \dots, p-2\}$ for which $i =  -\psi^k$. Therefore
\[
 a_{i,j} =  I[i = \phi^j -1],
\]
and so Projection $A$ is a $W_2(p, \phi)$ Gilbert-Welch Costas array by Theorem~\ref{thm:welch2}.

For Projection $C = (c_{j,k}) = (\sum_i d_{i,j,k})$, rewrite \eqref{eq:golomb-welch} as
\begin{equation}
\label{eq:rewrite}
 d_{i,j,k} = I[i  = -\psi^k \quad \mbox{and} \quad \phi^j + \psi^k =1 ]
\end{equation}
and similarly sum over $i$ to obtain
\[
 c_{j,k} = I[\phi^j + \psi^k =1],
\]
so that Projection $C$ is a $G_2(p, \phi, \psi)$ Golomb Costas array.

For Projection $B= (b_{i,k}) = (\sum_j d_{i,j,k})$, similarly sum \eqref{eq:rewrite} over $j$ to obtain 
\[
b_{i,k} = I[i = -\psi^k].
\]
The reflection $(r_{i,k})$ through a vertical axis of a $W_2(p, \psi)$ Gilbert-Welch Costas array $(s_{i,k})$ is given by 
\[
r_{i,k} = s_{p-1-i, k} = I[p-1-i = \psi^k -1] = I[i = -\psi^k],
\]
and so the reflection $(r_{i,k})$ equals Projection $B = (b_{i,k})$.
\end{proof}

\begin{example}
We construct a Costas cube $D = (d_{i,j,k})$ of order $11$ according to Theorem~\ref{thm:gw2w2cube}, using $p=13$, $\phi=11$ and $\psi=6$. From \eqref{eq:golomb-welch}, the triples $(i,j,k)$ for which $d_{i,j,k} = 1$ are given by
\[
\begin{array}{*{13}{c}}
$i$ & & 1 & 2 & 3 & 4 & 5 & 6 & 7 & 8 & 9 &10 &11 \\
$j$ & & 7 & 4 & 2 & 3 &11 & 5 & 9 & 8 &10 & 1 & 6 \\
$k$ & & 6 &11 & 2 & 4 & 3 & 7 & 1 & 9 &10 & 8 & 5 
\end{array}
\]
and the permutations corresponding to Projections $A$, $B$, $C$ of $D$ are given by
\[
\begin{array}{*{13}{c}}
$A$ & &10 & 3 & 4 & 2 & 6 &11 & 1 & 8 & 7 & 9 & 5 \\
$B$ & & 7 & 3 & 5 & 4 &11 & 1 & 6 &10 & 8 & 9 & 2 \\
$C$ & & 9 & 2 &11 & 3 & 6 & 7 & 5 & 1 & 8 &10 & 4 
\end{array}
\]
\end{example}

We now give our third and fourth algebraic constructions of an infinite family of Costas cubes. Each of Projections $A$, $B$, $C$ in both of these constructions is (equivalent to) a $G_3$ Golomb Costas array.
\begin{theorem}
\label{thm:g3g3g3cube}
Let $q>3$ be a prime power, and suppose there exists a primitive element $\phi$ of $\F_q$ for which both $1-\phi$ and $1-\phi^{-1}$ are also primitive. Then 
\begin{enumerate}[(i)]
\item
the array $D = (d_{i,j,k})$ over $\Z_2$ given by 
\begin{align}
  d_{i,j,k} = & I[
  \phi^{i+1} + (1-\phi)^{j+1} = 1
  \quad\mbox{and}\quad 
  \phi^{-(i+1)} + (1-\phi^{-1})^{k+1} = 1	
  \quad\mbox{and}				\nonumber \\
  & \phantom{I[} (1-\phi)^{-(j+1)} + (1-\phi^{-1})^{-(k+1)} = 1] 		
  \hspace{15mm} \mbox{for $i$, $j$, $k \in \{1, 2, \dots, q-3\}$} \label{eq:g3a}
\end{align}
is an order $q-3$ Costas cube for which 
Projection $A$ is a $G_3(q, \phi)$ Golomb Costas array, 
Projection $B$ is a $G_3(q, \phi^{-1})$ Golomb Costas array, and 
Projection $C$ is a $G_3(q,(1-\phi)^{-1})$ Golomb Costas array.
\item
the array $E = (e_{i,j,k})$ over $\Z_2$ given by 
\begin{align}
  e_{i,j,k} = & I[
  \phi^{i+1} + (1-\phi)^{j+1} = 1
  \quad\mbox{and}\quad 
  \phi^i + (1-\phi^{-1})^{k+1} = 1
  \quad\mbox{and}				\nonumber \\
  & \phantom{I[} (1-\phi)^j + (1-\phi^{-1})^k = 1]			
  \hspace{20mm} \mbox{for $i$, $j$, $k \in \{1, 2, \dots, q-3\}$} \label{eq:g3b}
\end{align}
is an order $q-3$ Costas cube for which 
Projection $A$ is a $G_3(q, \phi)$ Golomb Costas array, 
Projection $B$ is the reflection through a vertical axis of a $G_3(q, \phi^{-1})$ Golomb Costas array, and 
Projection $C$ is the rotation through $180^\circ$ of a $G_3(q,(1-\phi)^{-1})$ Golomb Costas array.
\end{enumerate}
\end{theorem}
\begin{proof}
By Definition~\ref{defn:Costascube}, we must show that each of Projections $A$, $B$, $C$ for (i) and (ii) is the specified Costas array. 
\begin{enumerate}[(i)]
\item
Consider Projection $A = (a_{i,j}) = (\sum_k d_{i,j,k})$.
By \eqref{eq:identity}, any two of the three conditions in the indicator function of \eqref{eq:g3a} imply the third, so we may rewrite \eqref{eq:g3a} as
\[
  d_{i,j,k} = I[
  \phi^{i+1} + (1-\phi)^{j+1} = 1
  \quad\mbox{and}\quad 
  \phi^{-(i+1)} + (1-\phi^{-1})^{k+1} = 1].
\]
For $i, j \in \{1, 2,  \dots, q-3\}$, sum over $k$ to obtain
\[
 a_{i,j}= \sum_{k=1}^{q-3} I[
  \phi^{i+1} + (1-\phi)^{j+1} = 1
  \quad\mbox{and}\quad 
  \phi^{-(i+1)} + (1-\phi^{-1})^{k+1} = 1].
\]
Now $1-\phi^{-(i+1)} \in \F_q \setminus \{0,1,1-\phi^{-1}\}$, and $1-\phi^{-1}$ is primitive by assumption, so by \eqref{eq:star} there is exactly one value $k \in \{1, 2,  \dots, q-3\}$ for which $\phi^{-(i+1)} + (1-\phi^{-1})^{k+1} = 1$.
Therefore
\[
 a_{i,j} =  I[\phi^{i+1} + (1-\phi)^{j+1} = 1],
\]
and so Projection $A$ is a $G_3(q, \phi)$ Golomb Costas array by Theorem~\ref{thm:golomb3}.

Similar arguments apply to Projections
$B = (b_{i,k}) = (\sum_j d_{i,j,k})$ and 
$C = (c_{j,k}) = (\sum_i d_{i,j,k})$.
For Projection $C$, note that we may use \eqref{eq:identity} to rewrite the condition
$(1-\phi)^{-(j+1)} + (1-\phi^{-1})^{-(k+1)} = 1$ in the indicator function of \eqref{eq:g3a} as
$((1-\phi)^{-1})^{j+1} + (1 - (1-\phi)^{-1})^{k+1} = 1$.

\item
By \eqref{eq:identity}, any two of the three conditions in the indicator function of \eqref{eq:g3b} imply the third.
Let $i,j,k \in \{1,2,\dots,q-3\}$.
By similar arguments to those used to prove~(i), Projections $A=(a_{i,j})$, $B=(b_{i,k})$, $C=(c_{j,k})$ satisfy 
\begin{align*}
  a_{i,j} &= I[\phi^{i+1} + (1-\phi)^{j+1} = 1], \\
  b_{i,k} &= I[\phi^i + (1-\phi^{-1})^{k+1} = 1], \\
  c_{j,k} &= I[(1-\phi)^j + (1-\phi^{-1})^k = 1].
\end{align*}

Therefore Projection $A$ is a $G_3(q,\phi)$ Golomb Costas array by Theorem~\ref{thm:golomb3}.

The reflection $(r_{i,k})$ through a vertical axis of a $G_3(q,\phi^{-1})$ Golomb Costas array $(s_{i,k})$ is given by
\[
r_{i,k} = s_{q-2-i,k} = I[(\phi^{-1})^{q-1-i}+(1-\phi^{-1})^{k+1} = 1] 
= I[\phi^i+(1-\phi^{-1})^{k+1} = 1] = b_{i,k}, 
\]
and so Projection $B$ equals the reflection~$(r_{i,k})$.

The rotation $(r'_{j,k})$ through $180^\circ$ of a $G_3(q,(1-\phi)^{-1})$ Golomb Costas array $(s'_{j,k})$ is given by
\[
r'_{j,k} = s'_{q-2-j,q-2-k} = I[((1-\phi)^{-1})^{q-1-j}+(1-(1-\phi)^{-1})^{q-1-k} = 1].
\]
Therefore by \eqref{eq:identity},
\[
r'_{j,k} = I[(1-\phi)^j + (1-\phi^{-1})^k = 1] = c_{j,k},
\]
and so Projection~$C$ equals the rotation~$(r'_{j,k})$.
\end{enumerate}
\end{proof}

\begin{example}
\label{ex:decube}
We construct Costas cubes $D = (d_{i,j,k})$ and $E = (e_{i,j,k})$ of order $24$ according to Theorem~\ref{thm:g3g3g3cube}, using $q=27$. Represent $\F_{3^3}$ as $\Z_3[x]/\langle 1+2x^2+x^3 \rangle$ and take $\phi = 2+2x$, for which $1-\phi = 2+x$ and $1-\phi^{-1} = x+x^2$ are also primitive. 
 From \eqref{eq:g3a}, the triples $(i,j,k)$ for which $d_{i,j,k} = 1$ are given by
\[
\begin{array}{*{26}{c @{\hspace{2.7mm}}}}
$i$ & & 1 & 2 & 3 & 4 & 5 & 6 & 7 & 8 & 9 &10 &11 &12 &13 &14 &15 &16 &17 &18 &19 &20 &21 &22 &23 &24\\
$j$ & & 6 & 2 & 4 & 7 &20 &21 & 3 & 8 &18 &15 &14 &12 & 5 &23 &17 &24 &10 &19 & 9 &13 & 1 &16 &11 &22\\
$k$ & &21 & 2 & 7 & 3 &13 & 1 & 4 & 8 &19 &17 &23 &12 &20 &11 &10 &22 &15 & 9 &18 & 5 & 6 &24 &14 &16
\end{array}
\]
and the permutations corresponding to Projections $A$, $B$, $C$ of $D$ are given by
\[
\begin{array}{*{26}{c @{\hspace{2.72mm}}}}
$A$ & &21 & 2 & 7 & 3 &13 & 1 & 4 & 8 &19 &17 &23 &12 &20 &11 &10 &22 &15 & 9 &18 & 5 & 6 &24 &14 &16\\
$B$ & & 6 & 2 & 4 & 7 &20 &21 & 3 & 8 &18 &15 &14 &12 & 5 &23 &17 &24 &10 &19 & 9 &13 & 1 &16 &11 &12 \\
$C$ & &21 & 2 & 7 & 3 &13 & 1 & 4 & 8 &19 &17 &23 &12 &20 &11 &10 &22 &15 & 9 &18 & 5 & 6 &24 &14 &16
\end{array}
\]
 From \eqref{eq:g3b}, the triples $(i,j,k)$ for which $e_{i,j,k} = 1$ are given by
\[
\begin{array}{*{26}{c @{\hspace{2.49mm}}}}
$i$ & & 1 & 2 & 3 & 4 & 5 & 6 & 7 & 8 & 9 &10 &11 &12 &13 &14 &15 &16 &17 &18 &19 &20 &21 &22 &23 &24\\
$j$ & & 6 & 2 & 4 & 7 &20 &21 & 3 & 8 &18 &15 &14 &12 & 5 &23 &17 &24 &10 &19 & 9 &13 & 1 &16 &11 &22\\
$k$ & &16 &14 &24 & 6 & 5 &18 & 9 &15 &22 &10 &11 &20 &12 &23 &17 &19 & 8 & 4 & 1 &13 & 3 & 7 & 2 &21
\end{array}
\]
and the permutations corresponding to Projections $A$, $B$, $C$ of $E$ are given by
\[
\begin{array}{*{26}{c @{\hspace{2.39mm}}}}
$A$ & &21 & 2 & 7 & 3 &13 & 1 & 4 & 8 &19 &17 &23 &12 &20 &11 &10 &22 &15 & 9 &18 & 5 & 6 &24 &14 &16\\
$B$ & &19 &23 &21 &18 & 5 & 4 &22 &17 & 7 &10 &11 &13 &20 & 2 & 8 & 1 &15 & 6 &16 &12 &24 & 9 &14 & 3\\
$C$ & & 9 &11 & 1 &19 &20 & 7 &16 &10 & 3 &15 &14 & 5 &13 & 2 & 8 & 6 &17 &21 &24 &12 &22 &18 &23 &4
\end{array}
\]
\end{example}
A primitive element $\phi$ for which both $1-\phi$ and $1-\phi^{-1}$ are also primitive in $\F_q$ (as required in Theorem~\ref{thm:g3g3g3cube}) does not necessarily exist for a prime power $q$; for example, there is no such primitive element in $\F_{2^4}$.
If such a $\phi$ exists, then the order $q-2$ Costas cube $(f_{i,j,k})$ constructed in Theorem~\ref{thm:gggcube} with $(\phi,\rho,\psi) = (\phi,(1-\phi)^{-1},1-\phi^{-1})$ satisfies $f_{1,1,1}=1$, by \eqref{eq:identity}. The order $q-3$ Costas cube $(d_{i,j,k})$ of Theorem~\ref{thm:g3g3g3cube}~(i) can then be viewed as arising from the removal of the three planes of~$(f_{i,j,k})$ containing the position $(1,1,1)$ to form $(d_{i,j,k})=(f_{i+1,j+1,k+1})$.
Furthermore, the Costas cube $(e_{i,j,k})$ of Theorem~\ref{thm:g3g3g3cube}~(ii) can be obtained from $(d_{i,j,k})$ by the rule 
\[
e_{i,j(i),k(i)}=1 \quad \mbox{if and only if} \quad d_{i, j(i), k(q-2-i)} = 1
\]
(as illustrated in Example~\ref{ex:decube}). 
Application of this rule does not necessarily preserve the Costas cube property,
and so $(d_{i,j,k})$ and $(e_{i,j,k})$ are inequivalent Costas cubes in general (even though Theorem~\ref{thm:g3g3g3cube} shows that Projections $A$ of these two cubes are identical, Projections B are equivalent Costas arrays, and Projections C are equivalent Costas arrays).

Table~\ref{tab:constructions} displays, for $n \le 29$, the number of equivalence classes of Costas cubes of order~$n$ constructed by Theorems~\ref{thm:gggcube}, \ref{thm:gw2w2cube}, and~\ref{thm:g3g3g3cube}, and for comparison the total number of equivalence classes of Costas cubes of order~$n$. By direct verification, all equivalence classes of Costas cubes of order at most 29 for which all three Projections $A$, $B$, $C$ are $G_2$ Golomb Costas arrays are constructed by Theorem~\ref{thm:gggcube}; all equivalence classes of Costas cubes of order at most 29 for which two of Projections $A$, $B$, $C$ are equivalent to $W_2$ Gilbert-Welch Costas arrays and the third Projection is a Golomb Costas array are constructed by Theorem~\ref{thm:gw2w2cube}; and all equivalence classes of Costas cubes of order greater than 2 and at most 29 for which all three Projections $A$, $B$, $C$ are equivalent to $G_3$ Golomb Costas arrays are constructed by Theorem~\ref{thm:g3g3g3cube}.

Table~\ref{tab:constructions} shows that the existence of Costas cubes of order 2, 3, 4, 20, 21, 24, 25, and 27 is completely explained by Theorems~\ref{thm:gggcube},\ref{thm:gw2w2cube}, and~\ref{thm:g3g3g3cube}. It would be interesting to find an explanation for the existence of the equivalence classes of Costas cubes not constructed by these three theorems, which would allow us to explain the existence of those associated Costas array projections that are currently regarded as sporadic.

\begin{table}[h]
\centering
\begin{tabular}{|c|c|c|c|c|}
\hline
Order 		& \# equivalence 		& \# equivalence 		& \# equivalence 		& Total \# 		\\ 
\hspace{4em}	& classes of Costas 		& classes of Costas 		& classes of Costas 		& equivalence classes 	\\
		& cubes constructed 		& cubes constructed 		& cubes constructed 		& of Costas cubes	\\
		& by Theorem~\ref{thm:gggcube}	& by Theorem~\ref{thm:gw2w2cube}& by Theorem~\ref{thm:g3g3g3cube}&			\\ \hline
2		& 1 				&   				&				& 1			\\
3		& 1 				& 1  				&				& 1			\\
4		&				&				& 2				& 2			\\
5		& 1				& 1				& 2				& 13			\\
6		& 4 				&  				&				& 47			\\
7		& 2 				&  				&				& 30			\\
9		& 4 				& 3				&				& 46			\\
11		& 4  				& 3				&				& 66			\\
14		& 5  				&  				&				& 6			\\
15		& 20  				& 10				&				& 33			\\
17		& 10 				& 6				&				& 19			\\
20		&				&				& 2				& 2			\\
21		& 35 				& 15				&				& 50			\\
23		& 10 				&  				&				& 11			\\
24		&				&				& 2				& 2			\\
25		& 20				&  				&				& 20			\\
27		& 56				& 21 				&				& 77			\\
29		& 20				& 10 				& 2				& 33			\\ \hline
\end{tabular}
\caption{Inequivalent Costas cubes constructed by Theorems~\ref{thm:gggcube}, \ref{thm:gw2w2cube}, and~\ref{thm:g3g3g3cube} (the order 3 Costas cube arises under Theorem~\ref{thm:gggcube} and also under Theorem~\ref{thm:gw2w2cube})}
\label{tab:constructions}
\end{table}

\section*{Acknowledgements}
We are grateful to Ladislav Stacho and Luis Goddyn, whose insightful questions and comments at the SFU Discrete Mathematics seminar in March 2017 led us to discover Theorem~\ref{thm:g3g3g3cube}. We thank the referees for their helpful comments.

\end{document}